\documentclass[12pt]{amsart}
\usepackage{amsmath,amssymb,amsfonts,amsthm,amstext,epsf}
\usepackage{url,xspace,enumerate,color,graphicx,verbatim}

\newcommand{\Z}{{\mathbb Z}}
\newcommand{\alt}{\text{\rm alt}}
\newcommand{\Lalt}{\vec\LL^2_\alt}
\newcommand{\palt}{p_\alt}
\newcommand\LL{{\mathbb L}}

\renewcommand{\P}{{\mathbb P}}
\newcommand{\E}{{\mathbb E}}

\newcommand{\dof}{\it} 
\newcommand\ol{\overline}
\newcommand\oo{\infty}
\newcommand\sm{\setminus}
\newcommand\De{\Delta}

\newcommand\rad{\text{\rm rad}}
\renewcommand\b{\beta}
\newcommand\sF{{\mathcal F}}
\newcommand\resp{respectively}

\renewcommand\a{\alpha}
\renewcommand\t{\theta}

\renewcommand\H{L}

\newcommand\es{\varnothing}
\newcommand\Om{\Omega}
\newcommand\om{\omega}
\newcommand\s{\sigma}

\newcommand\df[1]{\emph{#1}}

\newcommand\tH{\widetilde H}

\newcommand\oH{\overline H}
\newcommand\Du{\De_\text{\rm u}}

\renewcommand\o{\text{\rm o}}
\newcommand\lam{\lambda}

\newcommand\axmn{A^{x,m}(n)}
\newcommand\sD{{\mathcal D}}
\newcommand\nneg{non-negative}
\newcommand\as{a.s.}

\newenvironment{letlist}{\begin{list}{\rm(\alph{mycount})}%
   {\usecounter{mycount}\labelwidth=1cm\itemsep 0pt}}{\end{list}}

\newcommand\mlra{\stackrel{M}{\longleftrightarrow}}
\newcommand\klra{\stackrel{K}{\longleftrightarrow}}
\newcommand{\fr}{\rightarrowtail}
\newcommand{\di}{\Lambda}
\newcommand\eps{\epsilon}
\newcommand\dil{\lambda}

\newcommand\rap{\stackrel{+}{\fr}}
\newcommand\ram{\stackrel{-}{\fr}}
\newcommand\frp[1]{\rap_#1}
\newcommand\frl{\fr_\lambda}
\newcommand\rapl{\rap_\lambda}
\newcommand\hd{d_M}
\newcommand\kd{d_K}
\newcommand\qq{\qquad}
\newcommand\q{\quad}
\newcommand\pL{p_{\text{\rm L}}}
\newcommand\pM{p_{\text{\rm M}}}
\newcommand\down{\kern5pt\downarrow\kern-4pt}
\newcommand\bigmid{\,\big|\,}
\newcommand\Lip{\text{\sc lip}}
\newcommand\pco{\vec p_{\text{\rm c}}}
\newcommand\pca{p_{\text{\rm c}}^\downarrow}
\newcommand\pc{p_{\text{\rm c}}}

\newtheorem{thm}{Theorem}
\newtheorem{lemma}[thm]{Lemma}
\newtheorem{prop}[thm]{Proposition}
\newtheorem{cor}[thm]{Corollary}

\newcounter{mycount}

\begin{document}
\title{Geometry of Lipschitz percolation}

\author[Grimmett]{G.\ R.\ Grimmett}
\address[G.\ R.\ Grimmett]{Statistical Laboratory,
Centre for Mathematical Sciences, Cambridge University,
Wilberforce Road, Cambridge CB3 0WB, UK}
\email{g.r.grimmett@statslab.cam.ac.uk}
\urladdr{http://www.statslab.cam.ac.uk/$\sim$grg/}

\author[Holroyd]{A.\ E.\ Holroyd}
\address[A.\ E.\ Holroyd]{\sloppypar Microsoft Research,
1 Microsoft Way, Redmond WA 98052, USA; and Department of
Mathematics, University of British Columbia, 121-1984
Mathematics Road, Vancouver, BC V6T 1Z2, Canada} \email{holroyd
at microsoft dot com}
\urladdr{http://research.microsoft.com/$\sim$holroyd/}

\date{\today}


\begin{abstract}
We prove several facts concerning Lipschitz percolation,
including the following. The critical probability $\pL$ for the
existence of an open Lipschitz surface in site percolation on
$\Z^d$ with $d\ge 2$ satisfies the improved bound $\pL\le
1-1/[8(d-1)]$. Whenever $p>\pL$, the height of the lowest
Lipschitz surface above the origin has an exponentially
decaying tail.  The lowest surface is dominated stochastically
by the boundary of a union of certain independent, identically
distributed random subsets of $\Z^d$.  As a consequence, for
$p$ sufficiently close to $1$, the connected regions of
$\Z^{d-1}$ above which the surface has height $2$ or more
exhibit stretched-exponential tail behaviour.
\end{abstract}

\maketitle

\section{Lipschitz percolation}\label{sec:intro}

We consider site percolation with parameter $p$ on the lattice $\Z^d$ with $d
\ge 2$, with law denoted $\P_p$. The existence of open Lipschitz surfaces was
investigated in \cite{DDGHS}, the main theorem of which may be summarized as
follows. Let $\|\cdot\|$ denote the $\ell^1$-norm, and write
$\Z^+=\{1,2,\dots\}$ and $\Z^+_0=\{0\}\cup\Z^+$. A function $F:
\Z^{d-1}\to\Z^+$ is called {\dof Lipschitz} if:
\begin{equation}
\label{defLip0}
\text{for any } x,y\in \Z^{d-1} \text{ with } \|x-y\|=1,
\text{ we have } |F(x)-F(y)|\leq 1.
\end{equation}
Let $\Lip$ be the event that there exists a Lipschitz function
$F: \Z^{d-1}\to\Z^+$ such that,
\begin{equation}\label{defLip3}
\text{for each $x\in \Z^{d-1}$,
the site $(x,F(x))\in\Z^{d}$ is open.}
\end{equation}
The event $\Lip$ is clearly increasing. Since it is invariant
under translation of $\Z^{d}$ by the vector $(1,0,\dots,0)$,
its probability equals either $0$ or $1$. Therefore, there
exists $\pL \in [0,1]$ such that:
$$
\P_p(\Lip) = \begin{cases} 0 &\text{if } p<\pL,\\
1 &\text{if } p>\pL.
\end{cases}
$$

It was proved in \cite{DDGHS} that $0<\pL<1$, and more concretely that $0 <
\pL \le 1-(2d)^{-2}$. As noted in \cite{DDGHS}, when $p>\pL$, there exists  a
Lipschitz function $F$ satisfying \eqref{defLip0} with the property that the
random field $(F(x): x \in \Z^{d-1})$ is stationary and ergodic. Applications
of these and related statements may be found in \cite{DDS,GH10b,GH10a}.

\section{Main results}\label{sec:mainr}
Our first result is
an improvement of the upper bound for $\pL$ of \cite{DDGHS}.

\begin{thm}\label{thm:pcbnd}
For $d \ge 2$ we have $\pL\le 1-[8(d-1)]^{-1}$.
\end{thm}

This is proved in Section \ref{sec:handm}. The complementary inequality
$$
\pL\ge 1-\frac{1+\o(1)}{2d}\q\text{as} \q d \to\oo
$$
is proved in Section \ref{sec:pcineq}, yielding
that $1/d$ is the correct order of magnitude for $1-\pL(d)$
in the limit as $d \to \oo$.

A Lipschitz function $F$ satisfying \eqref{defLip0} is called
{\dof open}. For any family $\sF$ of Lipschitz functions, the
minimum (or `lowest') function
$$
\ol F(x):= \min\{F(x): F \in \sF\}
$$
is Lipschitz also. If there exists an open Lipschitz function,
there exists necessarily a lowest such function, and we refer
to it as the `lowest open Lipschitz function'. We shall
sometimes use the term `Lipschitz surface' to describe the
subset $\{(x,F(x)):x\in \Z^{d-1}\}$ of $\Z^d$, for some
Lipschitz $F$.  See Figure \ref{fig:lip}. We emphasize that
Lipschitz functions are always required to take values in the
\df{positive} integers.
\begin{figure}
 \centering\includegraphics[width=0.8\textwidth]{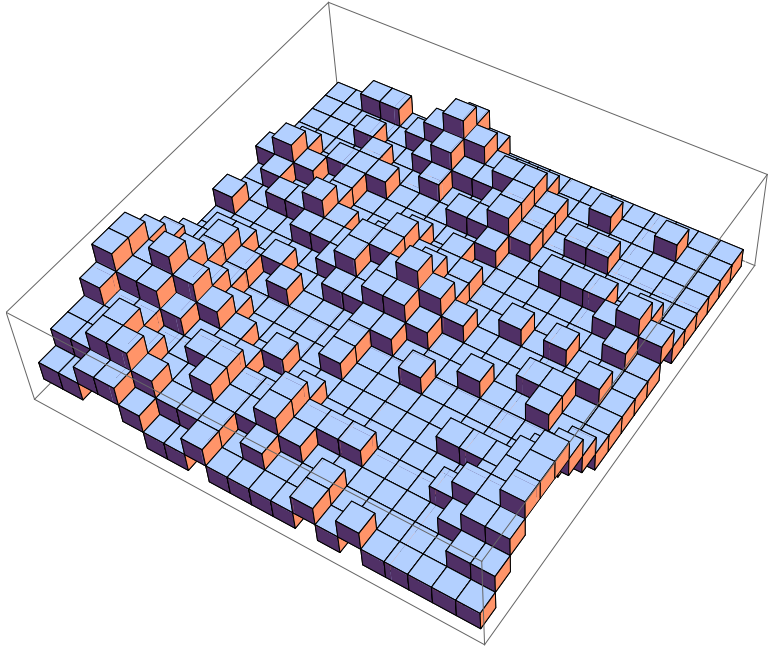}
\caption{Part of the lowest open Lipschitz surface
when $d=3$ and $p=0.18$. Each cube
represents an open site in the surface.}\label{fig:lip}
\end{figure}

For reasons of exposition, if there exists no open Lipschitz
function, we define the lowest open Lipschitz function by
$F(x)=\oo$ for all $x \in \Z^{d-1}$. Our second main result is
the following.

\begin{thm}\label{thm:expdec0}
Let $d\geq 2$ and let $F$ be the lowest open Lipschitz function.
There exists $\a=\a(d,p)$ satisfying $\a(d,p)>0$ for $p>\pL$ such that
\begin{equation}\label{eq:0dec}
\P_p(F(0)> n) \le e^{-\a n},\qq n\ge 0.
\end{equation}
\end{thm}

This is proved in Section \ref{sec:exp} by an adaptation of
Menshikov's proof of exponential decay for subcritical
percolation. Since the law of $F(x)$ is the same for all $x \in
\Z^{d-1}$, the choice of the origin $0$ in \eqref{eq:0dec} is
arbitrary. Theorem \ref{thm:expdec0} extends the
exponential-decay theorem of \cite{DDGHS} by removing the
condition on $p$ that is present in that work.

Our third result is a bound of a different type on the lowest open Lipschitz
function. It may be used to obtain information on the geometry of the
corresponding open surface. In preparation for this, we introduce the concept
of a local cover.  For $y\in \Z^{d-1}$, let $\sF_y$ be the set of functions
$f: \Z^{d-1} \to \Z^+_0$ that are Lipschitz in the sense of \eqref{defLip0},
and that satisfy:
\begin{letlist}
\item $f(y)>0$, and
\item for all $x\in\Z^{d-1}$, either $f(x)=0$ or the site $(x,f(x))$ is open.
\end{letlist}
The \df{local cover} at $y$ is the Lipschitz function $\ol f_y$ given by
$$
\ol f_y(z) := \min\{f(z): f\in \sF_y\}, \qq z\in \Z^{d-1}.
$$

It is shown in \cite{DDGHS} that the lowest open Lipschitz
function $F$ is given by
\begin{equation}\label{Fhill}
F(x) = \sup\{\ol f_y(x): y\in \Z^{d-1}\}.
\end{equation}
Now let $(g_y: y\in \Z^{d-1})$ be {\em independent} random
functions from $\Z^{d-1}$ to $\Z^+_0$ such that,  for each $y\in\Z^{d-1}$,
$g_y$ has the same law as $\ol f_y$. Let
\begin{equation}\label{Ghill}
G(x) := \sup\{g_y(x): y\in \Z^{d-1}\}.
\end{equation}

\begin{thm}\label{thm:indbnd0}
The lowest open Lipschitz function $F$ is dominated stoch\-astically by $G$
in that, for any increasing subset $A \subseteq [0,\oo]^{\Z^{d-1}}$,
$$
\P_p\Bigl[(F(x): x \in \Z^{d-1}) \in A\Bigr]
\le \P\Bigl[(G(x): x \in \Z^{d-1}) \in A\Bigr].
$$
\end{thm}

The proof of Theorem \ref{thm:indbnd0} is presented in Section
\ref{sec:decomp}. One application is the next result, a mild extension
of which is proved in
Section~\ref{sec:subcrith}.  Let $F$ be the lowest open Lipschitz function, and
recall that, by definition,  $F\geq 1$. Let $S$ be the set of all
$x\in\Z^{d-1}$ for which $F(x)>1$.  Let $S_0$ be the vertex-set of the
component containing $0$ in the subgraph of the nearest-neighbour lattice of
$\Z^{d-1}$ induced by $S$ (and take $S_0:=\es$ if $0\notin S$).

\begin{thm}\label{monument-simple}
Let $d\geq 2$.  There exists $\pM<1$ such that, for  $p>\pM$
and $\epsilon>0$,
\begin{equation}\label{bounds}
\exp\big(-\lambda n^{1/(d-1)}\bigr)\leq \P_p(|S_0|> n) \leq
\exp\big(-\gamma n^{1/(d-1)-\epsilon}\bigr),\quad n\geq 1,
\end{equation}
where $\lambda=\lambda(d,p)$ and $\gamma=\gamma(d,p,\epsilon)$ are positive
and finite.  If $d\neq 3$ then \eqref{bounds} holds even with $\epsilon=0$.
\end{thm}

The above statement is similar in spirit to Dobrushin's theorem \cite{Dob72}
concerning the existence of `flat' interfaces in the three-dimensional Ising
model with mixed boundary conditions (see also \cite{GG02} and \cite[Chap.\
7]{G-RC}). The proof utilizes a bound for the tail of the total progeny in a
subcritical branching processes with `stretched-exponential' family-sizes.
(The term `subexponential' is sometimes used in the literature.)
The $\epsilon$ of \eqref{bounds} arises from a certain instance in
the large-deviation theory of heavy-tailed distributions; see the discussion
of Section \ref{sec:subcrith}.

Section \ref{sec:betbnd} contains the basic estimate that leads in Section
\ref{sec:handm} to the proof of Theorem \ref{thm:pcbnd}. Lower bounds for the
critical value $\pL$ are found in Section \ref{sec:pcineq}. The
exponential-decay Theorem \ref{thm:expdec0} is proved in Section
\ref{sec:exp}, followed in Section \ref{sec:decomp} by the proof of Theorem
\ref{thm:indbnd0}. The final Section \ref{sec:subcrith} contains the proof of
Theorem \ref{monument-simple}.

\section{A basic estimate}\label{sec:betbnd}

This section contains a basic estimate (Proposition
\ref{thm:main3}) similar to the principal Lemma 3 of
\cite{DDGHS}, together with a lemma (Lemma \ref{lem2}) that
will be useful in the proof of Theorem \ref{thm:pcbnd}. We
begin by introducing some terminology.

Let $d\ge 2$ and $p=1-q \in [0,1]$. The site percolation model on $\Z^d$ is
defined as usual by letting each site $x\in\Z^d$ be {\dof open} with
probability $p$, or else {\dof closed}, with the states of different sites
independent. The sample space is $\Om=\{0,1\}^{\Z^d}$ where 1 represents
`open', and 0 represents `closed'. The appropriate product probability
measure is written $\P_p$, and expectation as $\E_p$. See \cite{G99} for a
general account of percolation.

As explained in \cite{DDGHS}, the lowest open Lipschitz
function $F$ may be constructed as a blocking surface to
certain paths. Let $e_1,e_2,\dots,e_d\in\Z^d$ be the standard
basis vectors of $\Z^d$. We define a \df{$\di$-path} from $u$
to $v$ to be any finite sequence of distinct sites
$u=x_0,x_1,\ldots,x_k=v$ of $\Z^d$ such that, for each
$i=1,2,\dots,k$,
\begin{equation}\label{lambdap}
x_i-x_{i-1}\in\{\pm e_d\}\cup \{- e_d\pm
    e_j:j=1,\ldots,d-1\}.
\end{equation}
The directed step $w:=x_i-x_{i-1}$ is called: \df{upward} (U)
if $w=e_d$; \df{downward vertical} (DV) if $w=-e_d$;
\df{downward diagonal} (DD) otherwise.

A $\di$-path is called \df{admissible} if
every upward step terminates at a closed site, which is to
say that, for each $i=1,2,\dots,k$,
\begin{equation} \label{defadm}
\text{if }x_i-x_{i-1} = e_d\text{ then $x_i$ is closed.}
\end{equation}

For $u= (u_1,u_2,\dots,u_d) \in \Z^d$, we write $h(u)=u_d$ for
its \df{height}.  Let
$$
\H:=\Z^{d-1}\times\{0\}\subset\Z^d$$
be the
hyperplane of height zero.

A $\di$-path is called a \df{$\dil$-path} if it has no downward
vertical steps. Denote by $u\fr v$ (\resp, $u\frl v$) the event
that there exists an admissible $\di$-path (\resp, $\dil$-path)
from $u$ to $v$. We write $u\rap v$ and $u \rapl v$ for the
corresponding events given in terms of paths using no vertex
$w\in\Z^d$ with $h(w)<0$. More generally, for $A,B\subseteq
\Z^d$ we write $A \fr B$ (and similarly for the other
relations) if $a \fr b$ for some $a\in A$ and $b\in B$.
Similarly, we write `$A \fr B$ in $C$' if such a path exists
using only sites in $C \subseteq \Z^d$.

\begin{prop}\label{thm:main3}
Let $d\ge 2$ and let $q=1-p$ satisfy $\rho:= 8(d-1)q<1$.
Then
\begin{equation}\label{now20}
\sum_{u\in\H:\, \|u\|\ge r} \P_p(0\rapl u)
\le \sum_{n=r}^\oo \binom{2n}n 2^{-2n} \rho^n, \qq r\ge 1.
\end{equation}
\end{prop}

\begin{proof}
Let $a=2d-2$. Any path contributing to the event $0\rapl u$
with $u\in \H$ uses some number $U$ of upward steps and some
number $D$ of downward diagonal steps. Furthermore, $U = D\ge
\|u\|$. Therefore,
\begin{align*}
\sum_{u\in\H:\, \|u\|\ge r} \P_p(0\rapl u)
&\le \sum_{U=D\ge r} \binom{U+D}U q^U a^D\\
&=\sum_{n=r}^\oo \binom {2n}n (qa)^n,
\end{align*}
as required.
\end{proof}

Since $\binom{2n}n \le 2^{2n}$,
it follows from Proposition \ref{thm:main3} when $\rho<1$ that
\begin{equation}\label{now21}
\sum_{u\in\H:\, \|u\|\ge r} \P_p(0\rapl  u) \le \frac{\rho^r}{1-\rho},
\qq r\ge 0.
\end{equation}
A marginally improved upper bound may be derived by using
either Stirling's formula or the local central limit theorem.

Here is a lemma concerning the relationship between $\di$-paths
and $\dil$-paths.
For $S\subseteq \Z^{d-1}\times\Z_0^+$, write
$$
\down S = \{x\in\Z^{d-1}\times\Z_0^+: x=s-ke_d
\text{ for some $s\in S$ and $k\ge 0$}\}.
$$

\begin{lemma}\label{lem2}
For $\om\in\Om$, we have that
$$
\{x\in\Z^{d-1}\times\Z_0^+: 0\rap x\} = \down
\{x\in\Z^{d-1}\times\Z_0^+: 0\rapl x\}.
$$
\end{lemma}

\begin{proof}
Since every $\dil$-path is a $\di$-path, and $\di$-paths may end with
any number of downward vertical steps without restriction, the right
side is a subset of the left side. It remains to show that the
left side is a subset of the right side.

Let $x\in \Z^{d-1}\times\Z^+_0$ be such that $0\rap x$, and let
$\pi$ be an admissible $\di$-path from $0$ to $x$ of shortest
length. If $\pi$ contains no DV step, then it is a $\dil$-path,
and we are done. Suppose there is a DV step in $\pi$, and
consider the last one, denoted $\ol V$, in the natural order of
the path. Then $\ol V$ is an ordered pair $(x,y)$ of sites of
$\Z^d$ with $y=x-e_d$. Since the sites of the path are
distinct, $\ol V$ is not followed by a U step. Therefore, $\ol
V$ is either at the end of the path $\pi$, or it is followed by
a DD step, which we write as the ordered pair $(y,z)$ with
$z=y-e_d+\a e_j$ for some $\a \in\{-1,1\}$ and $j \in
\{1,2,\dots,d-1\}$. In the latter case, we may interchange $\ol
V$  with this DD step, which is to say that the ordered triple
$(x,y,z)$ in $\pi$ is replaced by $(x,y',z)$ with $y'=y+\a
e_j$. This change does not alter the admissibility of the path
or its endpoints.

A small complication would arise if $y'\in \pi$. If this were
to hold, the site sequence thus obtained would contain a loop,
and we may erase this loop to obtain an admissible path from
$0$ to $x$  with fewer steps than $\pi$. This would contradict
the minimality of the length of $\pi$, whence $y' \notin\pi$.

Proceeding iteratively, the position in the path of the last DV
step may be delayed until: either it is the last step of the
path, or it precedes a U step, denoted $\ol U$. In the latter
case, we may remove $\ol U$ together with the previous DV step
to obtain a new admissible $\di$-path from $0$ to $x$ of
shorter length than $\pi$, a contradiction. Proceeding thus
with every DV step of $\pi$, we arrive at an admissible
$\di$-path from $0$ to $x$ comprising an admissible $\dil$-path
followed by a number of DV steps. The claim follows.
\end{proof}

\section{Hills, mountains, and the proof of Theorem \ref{thm:pcbnd}}
\label{sec:handm}

We describe next the use of Proposition \ref{thm:main3} in the
proof of Theorem~\ref{thm:pcbnd}. For $y\in \H$, the
\emph{hill} $H_y$ is given by
\begin{equation}\label{defhill}
H_y:=\{z\in \Z^d: y\rap z\}.
\end{equation}
Hills combine as follows to form `mountains'.
For $x\in \H$, the \emph{mountain} $M_x$ of $x$
is given by
\begin{equation}\label{defmtn}
M_x=\bigcup\{H_y: y\in\H \text{ such that } x\in H_y\}.
\end{equation}
Let $x\in \H$ and $S \subseteq \Z^d$. The \df{height} of $S$ at
$x$ is defined as
$$
l_x(S)=\sup\{k: x+ke_d \in S\},
$$
where the supremum of the empty set is
interpreted as $0$. The {\dof local height} of the mountain
$M_x$ is defined as its height $l_x(M_x)$ above $x$. By the
definition of $\di$-paths, we have that: either $l_x(M_x)<\oo$
for all $x\in \H$, or $l_x(M_x)=\oo$ for all $x$.

Define the event $I=\bigcap_{x\in \H}\{l_x(M_x)<\oo\}$. The
event $I$ is invariant under the action of translation of
$\Z^d$ by any vector in $\H$, whence it has probability either
$0$ or $1$. By the above, $\P_p(I)=1$ if and only if
$\P_p(l_0(M_0)<\oo)=1$.

Let the function
$F: \Z^{d-1}\to\Z^+\cup\{\oo\}$ be
defined by
\begin{equation}\label{defLip}
F(x) = 1+ l_x(M_x), \qq x\in \H,
\end{equation}
as in \eqref{Fhill}.  By the above discussion, $F$ is
finite-valued if and only if $I$ occurs. It is explained in
\cite{DDGHS} that $F$ is the lowest open Lipschitz function. In
conclusion, the lowest open Lipschitz surface is finite if and
only if $\P_p(I)=1$. In particular,
\begin{equation}\label{pLdef2}
\pL= \inf\{p: \P_p(I)=1\}.
\end{equation}
The random field $(F(x): x \in \H)$ is stationary and ergodic
under the action of translation of $\H$ by any $e_j$ with
$j\in\{1,2,\dots,d-1\}$.

We show in the remainder of this section that $\P_p(I)=1$ under
the condition of the following lemma. For $S \subseteq \Z^d$,
the \df{radius} of $S$ with respect to $0\in\Z^d$ is given by
$$
\rad(S) = \sup\{\|s\|: s\in S\}.
$$

\begin{lemma}\label{lem1}
Let $d \ge 2$, and let $\rho:= 8(d-1)q < 1$. Then
\begin{equation}\label{now25}
\P_p(\rad(H_0)\ge r) \le \frac {\rho^r}{1-\rho},\qq r\ge 1,
\end{equation}
and there exists an absolute constant $c=c(d)$ such that
\begin{equation}\label{now26}
\P_p(\rad(M_0)\ge r) \le cr^{d-1} \frac {\rho^{r/2}}{1-\rho},\qq r\ge 1.
\end{equation}
\end{lemma}

\begin{proof}[Proof of Theorem \ref{thm:pcbnd}]
Since $l_0(M_0) \le \rad(M_0)$, we have by \eqref{now26} that $l_0(M_0)<\oo$
\as\ when $\rho<1$. By \eqref{pLdef2}, $\pL\le 1-1/[8(d-1)]$.
\end{proof}

The following corollary of Lemma \ref{lem1} will be used in Section
\ref{sec:subcrith}. The \df{footprint} $L(S)$ of a subset $S \subseteq \Z^d$
is its projection onto $\H$:
\begin{equation}\label{footprint}
L(S) := \{(s_1,s_2,\dots,s_{d-1},0): (s_1,s_2,\dots,s_d)\in S\}.
\end{equation}

\begin{cor}\label{cor:subexp}
Let $d \ge 2$, $p=1-q \in(0,1)$, and $\rho:= 8(d-1)q$.
\begin{letlist}
\item There exists $\a=\a(d,p)<\oo$ such that
$$
\P_p(|L(H_0)|\ge n) \ge \exp(-\a n^{1/(d-1)}),\qq n \ge 2.
$$
\item There exists $\b=\b(d,p)$ satisfying $\b(d,p)>0$ when $\rho<1$,
such that
$$
\P_p(|L(M_0)|\ge n) \le \exp(-\b n^{1/(d-1)}),\qq n \ge 2.
$$
\end{letlist}
\end{cor}

\begin{proof}[Proof of Lemma \ref{lem1}]
Let $\rho<1$ and $r \ge 1$. Suppose $0\rapl u$ where
$u\in \Z^{d-1}\times \Z^+_0$ and $\|u\|\ge r$.
By considering
all sites $v$ such that there is a $\dil$-path from $u$ to $v$
using downward diagonal steps only, there must exist $v\in\H$ with
$0\rapl v$ and $\|v\|\ge r$. Therefore, by \eqref{now21},
$$
\P_p(\rad(H_0^\lam)\ge r) \le
\sum_{u\in\H:\, \|u\|\ge r} \P_p(0\rapl  u) \le \frac{\rho^r}{1-\rho},
$$
where
$$
H_y^\lam := \{z\in \Z^d: y\rapl z\},\qq y\in \H.
$$
By Lemma \ref{lem2}, $H_0 = \down H_0^\lam$, and in particular
$\rad(H_0)=\rad(H_0^\lam)$. Inequality \eqref{now25} follows.

By the definition of $M_0$ and the triangle inequality,
$$
\P_p(\rad(M_0) \ge r) \le \sum_{y\in \H}
\P_p\bigl(0\in H_y,\, \rad(H_y)\ge r-\|y\|\bigr).
$$
The last sum equals
$$
\sum_{y\in \H} \P_p\bigl(y\in H_0,\, \rad(H_0) \ge r-\|y\|\bigr),
$$
which we split into two sums depending on whether or not
$\|y\|\le r/2$. The first such sum is no larger than
$cr^{d-1}\P_p(\rad(H_0)\ge r/2)$ for some constant $c$. By
Lemma \ref{lem2} and \eqref{now21}, the second satisfies
\begin{align*}
\sum_{y\in \H:\, \|y\|> r/2} \P_p(y\in H_0)
&= \sum_{y\in \H:\, \|y\|> r/2} \P_p(0 \rapl y)\\
&\le \frac{\rho^{r/2}}{1-\rho},
\end{align*}
as required.
\end{proof}

\begin{proof}[Proof of Corollary \ref{cor:subexp}]
(a) There exists $c>0$ such that,
if every site in $\{ke_d: 1\le k\le m\}$ is
closed, then $|L(H_0)| \ge cm^{d-1}$.
Therefore,
\begin{equation}\label{now50}
\P_p(|L(H_0)|\ge cm^{d-1}) \ge q^m,
\end{equation}
and the first claim follows.

\noindent(b)
There exists $c >0$ such that $|L(M_0)| \le c\,\rad(M_0)^{d-1}$,
and the second claim follows by \eqref{now26}.
\end{proof}

\section{Inequalities for Lipschitz critical points}\label{sec:pcineq}

By Theorem \ref{thm:pcbnd}, we have $1-\pL(d) \ge 1/[8(d-1)]$.
In this section, we derive further results concerning the
values $\pL(d)$.   In particular we prove a lower bound for
$\pL$ that implies that the correct order of magnitude of
$1-\pL(d)$ is indeed $1/d$ in the limit of large $d$.

Consider the hypercubic lattice $\LL^d$ with vertex-set $\Z^d$.  Let $\pi$ be
a (finite or infinite) directed self-avoiding path of $\LL^d$ with vertices
$x_1,x_2,\ldots$. We call $\pi$ \df{acceptable} if it contains no upward
steps, i.e., if $x_{i+1} - x_i \neq e_d$ for all $i$.

Consider site percolation on $\Z^d$. An acceptable path is
called \df{open} (\resp, \df{closed}) if all of its sites are open
(\resp, closed). Let $\pca(d)$ be
the critical probability for the existence of an infinite open
acceptable path from the origin.

\begin{prop}\label{thm:pcdbond} The sequence $(\pL(d): d\ge 2)$
is non-decreasing and satisfies $\pL(d) \ge 1-\pca(d)$.
\end{prop}

\begin{figure}
\centering
\begin{picture}(0,0)\put(0,73.5){$0$}\end{picture}
\includegraphics[width=0.25\textwidth]{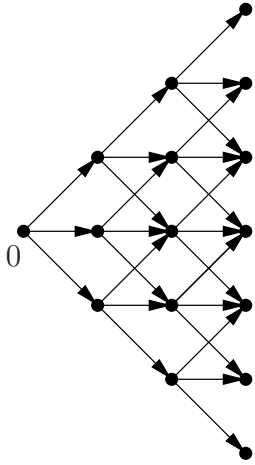}
\caption{Part of the lattice $\Lalt$ of \cite{GH}, obtained by
adding oriented edges to $\Z^2$.}\label{fig:2alt}
\end{figure}

Let $\pc(d)$ be the critical probability of site percolation on
$\LL^d$,  and let $\pco(d)$ be the critical probability of the
oriented site percolation process on $\LL^d$ in which every
edge is oriented in the direction of increasing coordinate
direction. It is elementary by graph inclusion that
$$
\pc(d) \le \pca(d) \le \min\{\pc(d-1),\pco(d)\}.
$$
Several proofs are known that $2d\pc(d) \to 1$ as $d \to \oo$
(see \cite[p.~30]{G99}; indeed the lace expansion permits an
expansion of $\pc(d)$ in inverse powers of $2d$). Hence, by
Proposition \ref{thm:pcdbond},
$$
1-\pL(d) \le \frac{1+\o(1)}{2d}\qq\text{as } d \to \oo.
$$

The value of $\pL(2)$ may be expressed as the critical value of
a certain percolation model. Consider the oriented graph
$\Lalt$ obtained from $\Z^2$ by placing an oriented bond from
$u$ to $v$ if and only if $v-u \in\{e_1, e_1\pm e_2\}$. This
graph was used in \cite{GH}, and is illustrated in Figure
\ref{fig:2alt}. Let $\palt$ be the critical probability of
oriented site percolation on $\Lalt$.  It is shown in \cite[Thm
5.1]{PF05} that $\palt\ge \frac12$. It is elementary that
$\palt\le \pco(2)$, and it was proved in \cite{BBS} that
$\pco(2) \le  0.7491$ (see also \cite{Lig95}). In summary,
$$
\tfrac12 \le \palt \le \pco(2) \le 0.7491.
$$
Simulations in \cite{PF05} indicate that $\palt \approx 0.535$.

\begin{prop}\label{thm:2dexact}
We have that $\pL(2) = \palt$.
\end{prop}

The equation $\pL(2) =\palt \ge \frac12$ has been noted in
independent work of Berenguer (personal communication).

\begin{proof}[Proof of Proposition \ref{thm:pcdbond}]
Let $F:\Z^{d-1}\to \Z^+$ be Lipschitz. The restriction
$G:\Z^{d-2}\to\Z^+$
given by
$$
G(x_1,x_2,\dots,x_{d-2}) := F(x_1,x_2,\dots,x_{d-2},0)
$$
is Lipschitz also, and the monotonicity of $\pL(d)$ follows.

Let $q=1-p$ satisfy $q>\pca(d)$. Suppose there exists an
acceptable closed path from $ne_d$ to some site in $\H$. Fix
such a path, and let $x$ be its earliest site lying in
$\H$.  Then we see that $x \rap
ne_d$, and hence $0 \in H_x$ and $l_0(M_0) \ge n$. By
\eqref{defLip}, $F(0) > n$.

On the other hand, by a standard argument of percolation
theory, on the event that there exists an infinite acceptable closed
path starting from $ne_d$,  there exists \as\ an
acceptable closed path from $ne_d$ to some site in $\H$.  Hence
$$
\P_p(F(0)>n) \ge \theta^\downarrow(q) > 0,
$$
where $\theta^\downarrow(q)$ is the probability that the origin
lies in an infinite acceptable closed path. Thus $p \le \pL$,
and hence $1-\pca \le \pL$.
\end{proof}

\begin{proof}[Proof of Proposition \ref{thm:2dexact}]
This is similar to part of the proof of \cite[Thm~6]{GH10a};
see also \cite[Lemma 7]{GH10a}. Let $d=2$. If $p>\pL$, there
exists \as\ a site $z$ on the $2$-coordinate axis of $\Z^2$
such that $z$ is the starting point  of some infinite open
oriented path of $\Lalt$. Hence $p \ge \palt$, so that $\pL \ge
\palt$.

Conversely, let $p>\palt$. By the block construction of
\cite{GH} or otherwise (see also \cite[Lemma 7]{GH10a}), there
is a strictly positive probability $\t^+(p)$ that the site
$(0,1)$ of $\Z^2$ is the starting point of an infinite open
oriented path of $\Lalt$ using no site with
$2$-coordinate lying in $(-\oo,0]$. By considering a reflection of $\Lalt$ in the
$2$-coordinate axis, there is probability at least
$p^{-1}\t^+(p)^2$ that there exists an open Lipschitz function
$F: \Z^2\to\Z^+$. Therefore, $p\ge \pL$, so that $\palt\ge
\pL$.
\end{proof}

\section{Exponential decay}\label{sec:exp}

In this section, we prove exponential decay of the tail of
$F(0)$ when $p>\pL$, as in Theorem \ref{thm:expdec0}. Let $d
\ge 2$ and $L_n=\Z^{d-1}\times \{n\}$, so that $\H=L_0$. For $x
\in \H$, let
$$
K_x:= \sup\{n: x\fr L_n\}.
$$
Recall the lowest open Lipschitz function $F$ of \eqref{defLip}.

\begin{lemma}\label{lemma:dual}
Let $p\in(0,1)$ and $x \in \H$.
The random variables $K_x$ and $F(x)-1 = l_x(M_x)$ have
the same distribution.
\end{lemma}

\begin{proof}
This holds by a process of path-reversal. It is convenient in
this proof to work with bond percolation rather than site
percolation. Each bond of the hypercubic lattice $\Z^d$ is
designated `open' with probability $p$ and `closed' otherwise,
different bonds receiving independent states. We call a
$\di$-path \df{admissible} if any directed step along a bond
from some $y$ to $y+e_d$ is closed. It is clear that the set of
admissible paths has the same law as in the formulation using
site percolation.

A  $\di^-$-path from $u$ to $v$ is defined as any
finite sequence of distinct sites $u=x_0,x_1,\ldots,x_k=v$ of
$\Z^d$ such that, for each $i=1,2,\dots,k$,
\begin{equation}\label{lambdap2-}
x_i-x_{i-1}\in\{\pm e_d\}\cup \{e_d\pm
    e_j:j=1,\ldots,d-1\}.
\end{equation}
Any step in the direction $-e_d$ is called \df{downward}.
A $\di^-$-path is called \df{$(-)$-admissible} if
every downward step traverses a closed bond.

Let $\pi$ be a $\di$-path of $\Z^{d}$, and let $\rho\pi$ be the
$\di^-$-path obtained by reversing each step. Note that $\pi$
is admissible if and only if $\rho\pi$ is $(-)$-admissible.

It suffices to assume $x=0$.
Let $\Pi_n$ be the union over $y \in L_0$ of the set of $\di$-paths from
$y$ to $ne_d$. Then $\rho\Pi_n$ is the set of
$\di^-$-paths from $ne_d$ to $L$, so that
$$
\P_p(L \fr ne_d) = \P_p(ne_d \ram L),
$$
where $\ram$ denotes connection by an admissible $\di^-$-path.
By a reflection of the lattice, the
last probability equals $\P_p(0\fr L_n)$, so that
$$
\P_p(l_0(M_0) \ge n) = \P_p(0\fr L_n),
$$
and the claim follows by \eqref{defLip}.
\end{proof}

\begin{thm}\label{thm:expdec}
There exists $\a=\a(d,p)$ satisfying $\a(d,p)>0$ for $p>\pL$ such that
$$
\P_p(0 \fr L_n) \le e^{-\a n},\qq n\ge 0.
$$
\end{thm}

\begin{proof}[Proof of Theorem \ref{thm:expdec0}]
This is immediate by Theorem \ref{thm:expdec} and Lemma
\ref{lemma:dual}.
\end{proof}

\begin{proof}[Proof of Theorem \ref{thm:expdec}]
The proof is an adaptation of Menshikov's proof of a
corresponding fact for percolation, as presented in \cite[Sect.
5.2]{G99}. We shall use the BK inequality of \cite{BK} (see
also \cite[Sect.\ 2.3]{G99}). The application of the BK
inequality (but not the inequality itself) differs slightly in
the current context from that of regular percolation, and we
illustrate this as follows. Let $a, b, u, v \in \Z^d$. The
`disjoint occurrence' of the events $\{a \fr b\}$ and $\{u \fr
v\}$ is written as usual $\{a\fr b\} \circ \{u\fr v\}$.  In
this setting, it comprises the set of configurations such that: there
exist admissible $\di$-paths $\pi_{a,b}$ from $a$ to $b$, and
$\pi_{u,v}$ from $u$ to $v$, such that each directed edge from
some $x$ to $x+e_d$ lies in no more than one of $\pi_{a,b}$,
$\pi_{u,v}$. That is, the paths must have no upward step in
common; they are permitted to have downward steps in common.

Let $m \ge 1$ and
$R_m=[-m,m]^{d-1} \times \Z$. For $x \in R_m \cap L$
and $r \ge 0$, let
$$
g_p^{x,m}(r) = \P_p(x \fr L_r\text{  in } R_m),
\qq g_p(r) = \P_p(0 \fr L_r),
$$
and
\begin{equation}\label{l1}
h_p^m(r) = \max\{g_p^{x,m}(r): x \in R_m \cap L\}.
\end{equation}
On recalling the definition of an admissible $\di$-path, we see
that the event $\{x \fr L_r\text{ in } R_m\}$ is a
finite-dimensional cylinder event, and so each $g_p^{x,m}(r)$ is a
polynomial in $p$. Note that
$$
h_p^m(r) \ge g_p^{0,m}(r) \to g_p(r)\q \text{ as } m \to\oo.
$$
Since $g_p^{x,m}(r) \le  g_p^{x,\oo}(r) = g_p(r)$, we have that
\begin{equation}\label{l2}
h_p^m(r) \to g_p(r) \qq\text{as } m\to\oo.
\end{equation}

By \eqref{l1}, $h_p^m(r)$ is the maximum of a finite
set of polynomials in $p$. Therefore, $h_p^m(r)$ is a continuous
function of $p$, and there exists a finite set
$\sD^m(r) \subseteq (0,1)$ such that:
for $p\notin \sD^m(r)$,
there exists $x=x^m_p(r) \in R_m \cap L$ with
\begin{equation}\label{l3}
\frac d{dp}h_p^m(r) = \frac d{dp} g_p^{x,m}(r).
\end{equation}

Let $x \in R_m \cap L$ and $\axmn= \{x\fr L_n \text{ in }
R_m\}$, so that $g_p^{x,m}(n) = \P_p(\axmn)$. By Russo's
formula, \cite[eqn (5.10)]{G99},
\begin{equation}\label{l5}
\frac d{dp} \log \P_p(\axmn)
= -\frac 1{1-p} \E_p\bigl(N(\axmn)\bigmid \axmn\bigr),
\end{equation}
where $N(A)$ is the number of pivotal sites for a decreasing
event $A$. We claim, as in \cite[eqn (5.18)]{G99},
that
\begin{equation}\label{k2}
\E_p(N\bigl(\axmn) \bigmid \axmn\bigr) \ge \frac n{\sum_{r=0}^n h_p^m(r)} -1.
\end{equation}
Once this is proved, it follows by \eqref{l5} that
$$
\frac d{dp}\log g_p^{x,m}(n)
\le -\frac 1{1-p}\left(\frac n{\sum_{r=0}^n h_p^m(r)} -1\right).
$$
Therefore, by \eqref{l3}, for $p\notin \sD^m(n)$,
\begin{equation}\label{l7}
\frac d{dp}\log h_p^m(n)
\le -\frac 1{1-p}\left(\frac n{\sum_{r=0}^n h_p^m(r)} -1\right).
\end{equation}

We integrate \eqref{l7} to obtain, for $0<\a < \b<1$,
$$
h_\b^m(n) \le h_\a^m(n)
\exp\left(-(\b-\a)\left[\frac n{\sum_{r=0}^n h_\a^m(r)} -1\right]\right),
$$
as in \cite[eqn (5.22)]{G99}. Let $m\to\oo$, and deduce by \eqref{l2} that
\begin{equation}\label{l6}
g_\b(n) \le g_\a(n)
\exp\left(-(\b-\a)\left[\frac n{\sum_{r=0}^n g_\a(r)} -1\right]\right).
\end{equation}
This last inequality may be analyzed just as
in \cite[Sect. 5.2]{G99} to obtain the claim of the theorem.

It remains to prove \eqref{k2}, and the proof is essentially
that of \cite[Lemma 5.17]{G99}. Fix $x \in R_m \cap L$, and let
$V$ be a random variable taking values in the non-negative
integers with
$$
\P(V \ge k) = h_p^{m}(k),\qq k \ge 0.
$$
Suppose $\axmn$ occurs, and let
$z_1,z_2,\dots,z_N$ be the pivotal sites for $\axmn$, listed
in the order in which they are encountered beginning at $x$.
Let $\rho_1:= \max\{0,h(z_1)-1\}$ and
\begin{equation}\label{defrho}
\rho_i:=
\max\bigl\{0,h(z_{i})-h(z_{i-1})-1\bigr\}, \qq i=2,3,\dots,N.
\end{equation}
Thus, $\rho_i$ measures the positive part of the
`vertical' displacement between the $(i-1)$th and $i$th pivotal sites.

We claim as in \cite[eqn (5.19)]{G99} that, for $k \ge 1$,
\begin{equation}\label{k2.5}
\P_p\bigl(\rho_1+\dots+\rho_k \le
 n-k \bigmid \axmn\bigr) \ge \P(V_1+\dots+V_k \le n-k)
\end{equation}
where the $V_i$ are independent, identically distributed copies
of $V$. Equation \eqref{k2} follows from this as in
\cite[Sect.\ 5.2]{G99}, and, as in that reference, it suffices
for \eqref{k2.5} to show the equivalent in the current setting
of \cite[Lemma 5.12]{G99}, namely the next lemma.
\end{proof}

\begin{lemma}
Let $k$ be a positive integer, and let $r_1,r_2,\dots,r_k$ be
non-negative integers with sum not exceeding $n-k$. With the
above notation, for $x\in R_m\cap L$,
\begin{align}\label{k3}
\P_p\bigl(\rho_k \le r_k,\ &\rho_i=r_i
\text{\rm\ for } 1\le i < k \bigmid \axmn\bigr)\\
&\ge \P(V\le r_k) \P_p\bigl(\rho_i=r_i
\text{\rm\ for } 1\le i < k\bigmid \axmn\bigr).\nonumber
\end{align}
\end{lemma}

\begin{proof}
Suppose first that $k=1$ and $0\le r_1 \le n-1$. By the BK
inequality,
\begin{align*}
\P_p(\{\rho_1>r_1\}\cap \axmn) &\le \P_p(A^{x,m}(r_1+1) \circ \axmn)\\
&\le \P_p(A^{x,m}(r_1+1)) \P_p(\axmn)\\
&\le \P(V > r_1)\P_p(\axmn),
\end{align*}
so that \eqref{k3} holds with $k=1$.

Turning to the general case, let $k\ge 1$ and let the $r_i$
satisfy the given conditions. For a site $z$, let $D_z$ be the
set of sites attainable from $x$ along admissible $\di$-paths
of $R_m$ not containing the upward step from $z-e_d$ to $z$.
Let $B_z$ be the event that the following statements hold:
\begin{letlist}
\item $z-e_d\in D_z$, and $z \notin D_z$,
\item $z$ is closed,
\item $D_z$ contains no vertex of $L_n$,
\item the pivotal sites for the event $\{ 0\fr z\}$ are, taken in
order, $z_1,z_2,
\dots, z_{k-1}=z$, with the $\rho_i$ of \eqref{defrho} satisfying
$\rho_i=r_i$ for $1\leq i<k$.
\end{letlist}
Let $B=\bigcup_z B_z$, and note that
\begin{equation}\label{then1}
B \cap \axmn = \{\rho_i=r_i \text{ for } 1\le i<k\} \cap \axmn.
\end{equation}
For $\om \in \axmn \cap B$, there exists a unique site
$\zeta=\zeta(\om)$ such that $B_\zeta$ occurs.

Let $\Delta$ denote the ordered pair $(D_\zeta,\zeta)$.
Now,
\begin{equation}\label{k4}
\P_p(\axmn \cap B) = \sum_\delta
\P_p(B\cap \{\De=\delta\})\P_p\bigl(\axmn\bigmid B\cap \{\De=\delta\}\bigr),
\end{equation}
where the sum is over all possible values $\delta=(\delta_z,z)$ of the
random pair $\De$. Note that
$$
\P_p\bigl(\axmn\bigmid B\cap \{\De=\delta\}\bigr)=
\P_p\bigl(z \fr L_n \text{ in } R_m\sm \delta_z\bigr).
$$

By a similar argument and the BK inequality,
\begin{align}\label{k5}
\P_p(&\{\rho_k>r_k\} \cap \axmn \cap B)\\
&= \sum_\delta \P_p(B\cap \{\De=\delta\})
\P_p\bigl(\{\rho_k>r_k\}\cap\axmn\bigmid B\cap \{\De=\delta\}\bigr),
\nonumber
\end{align}
and
\begin{align*}
\P_p\bigl(\{\rho_k>r_k\}&\cap\axmn\bigmid B\cap \{\De=\delta\}\bigr)\\
&\le \P(V >r_k) \P_p\bigl(\axmn \bigmid B\cap \{\De=\delta\}\bigr).
\end{align*}
On dividing \eqref{k5} by \eqref{k4}, we deduce that
$$
\P_p\bigl(\rho_k > r_k \bigmid \axmn \cap B\bigr) \le \P(V>r_k).
$$
The lemma is proved on multiplying through by $\P_p(B \mid \axmn)$,
and recalling \eqref{then1}.
\end{proof}

\section{Decomposition of hill-ranges}\label{sec:decomp}

This section is devoted to the domination argument used in the
proof of Theorem \ref{thm:indbnd0}. Let $\om\in \Om$ be a
configuration of the site percolation model on $\Z^d$. Let
$v_1,v_2,\dots$ be an arbitrary but fixed ordering of the
vertices in $\H$, and write $H_i=H_{v_i}$ for the hill at
$v_i$, as defined in \eqref{defhill}. We shall construct the
$H_i$ in an iterative manner, and observe the relationship
between a hill thus constructed and the previous hills. To this
end, we introduce some further notation. For $i \ge 1$, let
$$
\oH_i :=\bigcup_{j\le i} H_j.
$$

We call a $\di$-path {\dof \nneg} if it visits no site $w$ with
$h(w)<0$. For $u,v\in \Z^d \sm \oH_i$, write $u \frp i v$ if
there exists an admissible \nneg\ $\di$-path from $u$ to $v$
using no site of $\oH_i$. The `restricted' hill at $v_{i+1}$ is
given by
$$
\tH_{i+1} :=  \{z\in\Z^d: v_{i+1} \frp i z\}.
$$
That is, $\tH_{i+1}$ is given as before but in terms of \nneg\
$\di$-paths in the region obtained from $\Z^d$ by removal of
all hills already constructed. If $v_{i+1} \in \oH_i$, then
$\tH_{i+1} := \es$. Finally, let $\tH_1 := H_1$.

\begin{lemma}\label{thm:decomp}
For $\om\in\Om$,
\begin{equation}\label{now12}
\oH_i = \bigcup_{j\le i} \tH_j, \qq i\ge 1.
\end{equation}
\end{lemma}

Note that the above is a \emph{pointwise} statement in that it
holds for all configurations $\om$. Its main application is as
follows. In writing that two random variables $A$, $B$ may be
coupled with a certain property $\Pi$, we mean that there
exists a probability space that supports two random variables
$A'$, $B'$ having the same laws as $A$, $B$ and with property
$\Pi$. Let $(J_i: i=1,2,\dots)$ be independent, identically
distributed subsets of $\Z^d$ such that $J_i$ has the same law
as $H_i$.

\begin{thm}\label{thm:indbnd}
The families $(H_i:i \ge 1)$ and $(J_i: i \ge 1)$ may be
coupled in such a way that the following property holds:
\begin{equation}\label{now11}
\oH_i \subseteq \bigcup_{j\le i} J_j, \qq i \ge 1.
\end{equation}
\end{thm}

\begin{proof}[Proof of Lemma \ref{thm:decomp}]
We prove equation \eqref{now12} by induction on $i$.
It is trivial for $i=1$.

Suppose \eqref{now12} holds for $i=I\ge 1$, and consider the
case $i=I+1$. Suppose $\oH_I$ has been found, and consider the
vertex $v_{I+1}$. If $v_{I+1} \in \oH_I$, then
$H_{I+1}\subseteq \oH_I$ and $\tH_{I+1}=\es$. In this case,
\eqref{now12} holds with $i=I+1$.

Suppose $v_{I+1} \notin \oH_I$. Since $\tH_{I+1}$ is the hill
of $v_{I+1}$ within a restricted domain, we have $\tH_{I+1}
\subseteq H_{I+1}$ and $\oH_{I+1} \supseteq \oH_I\cup
\tH_{I+1}$. It remains to prove that $\oH_{I+1} \subseteq
\oH_I\cup \tH_{I+1}$, which holds if $\oH_{I+1} \sm \oH_I
\subseteq  \tH_{I+1}$. Let $y \in \oH_{I+1}\sm \oH_I =
H_{I+1}\sm \oH_I$. Since $y \in H_{I+1}$, there exists an
admissible \nneg\ $\di$-path $\pi$ from $v_{I+1}$ to $y$. If
$\pi \cap \oH_I \ne \es$, $\pi$ has a first vertex $z$ lying in
$\oH_I$. Since no admissible \nneg\ $\di$-path can exit
$\oH_I$, all points on $\pi$ beyond $z$ belong to $\oH_I$, and
in particular $y\in \oH_I$, a contradiction. Therefore, there
exists an admissible \nneg\ $\di$-path from  $v_{I+1}$ to $y$
not intersecting $\oH_I$, which is to say that $y \in
\tH_{I+1}$. In this case, \eqref{now12} holds with $i=I+1$.

In all cases, \eqref{now12} holds with $i=I+1$, and the
induction step is complete.
\end{proof}

\begin{proof}[Proof of Theorem \ref{thm:indbnd}]
For $S \subseteq\Z^d$, write
$$
\Du S=\{x+e_d: x\in S\}\sm S.
$$
By the definition of the hill $H_y$ of a vertex $y\in\H$, every
vertex in $\Du H_y$ is open. Since $H_y$ may be constructed by
a path-exploration process, an event of the form $\{H_y= S\}$
is an element of the $\s$-field $\sF_S$ generated by the random
variables $\om(s)$, $s \in S \cup \Du S$. Furthermore, for $i
\ge 1$, the event $\{\oH_i=S\}$ lies in $\sF_S$.

Suppose we are given that $\oH_i = S$ for some $i$ and some
$S$. Any event defined in terms of admissible \nneg\
$\di$-paths of $\Z^d \sm S$ is independent of the states of
$\Du S$, since no such admissible $\di$-path contains an upward
step with second endvertex in $\Du S$.

Consider a sequence of independent site percolations on $\Z^d$
with parameter $p$. Let $J_i$ be the hill at $v_i$ in the $i$th
such percolation model. In particular, for every $i$, $J_i$ has
the same law as $H_i$. We construct as follows a sequence
$(\tH_i')$ with the same joint law as $(\tH_i)$ and satisfying
$$
\bigcup_{j\le i} \tH_j' \subseteq \bigcup_{j\le i} J_j,\qq i \ge 1,
$$
and the claim of the theorem will follow by Lemma
\ref{thm:decomp}. First, we take $\tH_1'=J_1$. Then we let
$\tH_2'$ be the subset of $J_2$ containing all endpoints of all
admissible \nneg\ $\di$-paths from $v_2$ in the second
percolation model that do not intersect $\tH_1'$. More
generally, having found $\tH_j'$ for $j<i$, we let $\tH_i'$ be
the subset of $J_i$ containing all endpoints of all admissible
\nneg\ $\di$-paths from $v_i$ in the $i$th percolation model
that do not intersect $\tH_1'\cup\cdots\cup \tH_{i-1}'$.
\end{proof}

\begin{proof}[Proof of Theorem \ref{thm:indbnd0}]
By Theorem \ref{thm:indbnd}, there exists a probability space
on which are defined random variables $(H_y', J_y': y \in
\H)$ such that
\begin{letlist}
\item the family $(H_y': y \in \H)$ has the same joint law
    as $(H_y: y\in\H)$,
\item the $J_y'$ are independent, and each $J_y'$ has the same law as $H_y$,
\item for all $x\in\H$,
$$
\bigcup\{H_y': y\in\H\text{ with } x\in H_y'\}
\subseteq \bigcup\{J_y': y\in\H \text{ with } x\in J_y'\}.
$$
\end{letlist}
Let
\begin{align*}
F'(x) &= 1+\sup\{l_x(H_y'): y\in \H\},\\
G'(x) &= 1+\sup\{l_x(J_y'): y\in \H\}.
\end{align*}
By (c) above, $F'(x) \le G'(x)$ for all $x\in\H$, and the claim
follows since $F'$ (\resp, $G'$) has the same law as $F$
(\resp, $G$).
\end{proof}

\section{Finiteness of Mountain-Ranges}\label{sec:subcrith}

We next apply Theorem \ref{thm:indbnd0} in order to prove
Theorem \ref{monument-simple}, which states in particular that
for $p$ sufficiently close to $1$, the lowest open Lipschitz
surface is simply the hyperplane $\H+e_d$ pierced by
mountain-ranges of finite extent. Moreover, we prove an upper
bound for the tail of the volume of a mountain-range.

The \df{footprint} $L(S)$ of a subset $S \subseteq \Z^d$
was defined at \eqref{footprint}.
Let $H_x$ (\resp, $M_x$) be the hill
(\resp, mountain) at the site $x \in \H$, as in
\eqref{defhill} (\resp, \eqref{defmtn}),
and note that $M_x \cap M_y \ne \es$ if and only
if $L(M_x) \cap L(M_y) \ne\es$.
For $x,y \in \H$ with $x\ne y$, write $x\mlra y$ if
$M_x\cap M_y \ne \es$. Let $G(\mlra)$ denote the graph having
vertex set $\H$, and an edge between vertices $x$ and $y$ if
and only if $x\mlra y$.
Connected components of $G(\mlra)$ are called
\emph{mountain-ranges}. The mountain-range $R_x$ at the vertex
$x \in \H$ is the set of all vertices $z \in \H$ such that:
there exists $k\ge 1$ and  $x_1,x_2,\dots,x_k\in \H$ such that
$x\in M_{x_1}$, $z \in M_{x_k}$, and $M_{x_i} \cap M_{x_{i+1}}
\ne \es$ for $1\le i < k$. For $z \in R_x$ with $z\ne x$, the
minimal value of $k$ above is denoted $\hd(x,z)$. Note that $x
\in R_x$, and set $\hd(x,x)=0$.

As a measure of the extent of the mountain-range at $x$, we
shall study its volume $|R_x|$ and its `mountain radius' given
by
$$
\rho_M(x): = \sup\{\hd(x,z): z\in R_x\}.
$$

The results of this section are valid for  $p$ sufficiently large. Let
$$
\sigma(d,p) := \sum_{r\ge 1} c(2r+1)^{d-2}r^{d-1}\frac{\rho^{r/2}}{1-\rho},
$$
where $\rho = 8(d-1)q<1$ and $c$ is given as in Lemma \ref{lem1}.
The function $\s$ is decreasing in $p$. Let
\begin{equation}\label{eqn:critmean}
\pM(d):= \inf\{p: \sigma(d,p) < 1\},
\end{equation}
and note that $\pM<1$. We shall work with $p>\pM$, and have not attempted to
weaken this assumption.

\pagebreak
\begin{thm}\label{thm:noperc}
Let $d \ge 2$ and $\pM<p<1$, so that $\s=\s(d,p)<1$.
\begin{letlist}
\item We have $\E_p|R_0|\le (1-\s)^{-1}$, and
$$
\P_p(\rho_M(0) \ge n) \le \s^{n}, \qq n \ge 0.
$$
\item
Let $d \ne 3$.
There exists $\gamma=\gamma(d,p)>0$
such that
$$
\P_p(|R_0| \ge n) \le \exp(-\gamma n^{1/(d-1)}), \qq n \ge 2.
$$
\item Let $d =3$. For every $\epsilon>0$, 
    there exists $\gamma=\gamma(p,\epsilon)>0$ such that
$$
\P_p(|R_0| \ge n) \le \exp(-\gamma n^{\frac12-\epsilon}), \qq n \ge 2.
$$
\end{letlist}
\end{thm}

A slightly more precise estimate may be obtained when $d=3$, but we omit this
since our methods, in their simplest form, will not deliver the anticipated
tail $\exp(-\gamma n^{1/2})$. The reason for this is that the exponent
$n^{1/2}$ is a boundary case of the large-deviation theory of random
variables with stretched-exponential tails, as explained for example in
\cite{DDS1}.
With the exception of this case, the bounds of Theorem \ref{thm:noperc}(c)
have optimal order; see Corollary \ref{cor:subexp}, and note that $L(H_0)
\subseteq R_0$.

\begin{proof}[Proof of Theorem \ref{monument-simple}]
From \eqref{defLip} we have $S_0\times\{0\}\subseteq R_0$, so the upper bound
is immediate from Theorem \ref{thm:noperc}(b,c).  The lower bound may be proved by a
minor modification of \eqref{now50}, or may be
deduced directly from Corollary \ref{cor:subexp}(a) as follows.  Let $Y$ be
the set of $y\in L(H_0)$ for which $l_y(H_0)\geq 1$.  Then
$S_0\times\{0\}\supseteq Y$, and, since every site in $L(H_0)\setminus Y$
must have a neighbour in $Y$, we have $|S_0|\geq |Y|\geq (2d-1)^{-1}
|L(H_0)|$.
\end{proof}

Theorem \ref{thm:noperc} is proved by a comparison with an
independent family, and an appeal to results for branching
processes. Let $d \ge 2$, and let $J=(J_y: y\in \H)$ be a
vector of random subsets of $\H$ such that:
\begin{letlist}
\item for each $y\in \H$ we have $y\in J_y$,
\item the sets $J_y$ are independent,
\item the distribution of the translated set $J_y - y$,
does not depend on the choice of $y$.
\end{letlist}
We shall impose a further condition on $J$, namely the following. We say that
$J$ is \df{$0$-monotone} if, for $S \subseteq \H$ with $0\notin S$, the
conditional distribution of $J_0$, given $S \cap J_0 = \es$, is
stochastically smaller than $J_0$.  Let $\P$ denote the appropriate
probability measure.

The random set $J_0$ satisfies the above condition whenever its law,
considered as a probability measure on $\{0,1\}^{\H}$, is positively
associated (a discussion of positive association may be found in \cite[Sect.\
2.2]{G-RC}). An example of this arises in a commonly studied instance of the
continuum percolation model, namely when $J_0$ has support in a given
\emph{increasing} sequence of subsets of $\H$.

For $x\in\H$ , let
\begin{equation}\label{defK}
K_x := \bigcup \{J_y: y\in\H \text{ is such that } x\in J_y\}.
\end{equation}
For $x,y \in \H$, we write $x \klra y$
if there exist $k \ge 1$ and $z_1,z_2,\dots,z_k\in\H$
such that $x\in K_{z_1}$, $y\in K_{z_k}$, and
$K_{z_i} \cap K_{z_{i+1}} \ne \es$
for $1\le i < k$.
When $z\ne x$, we write $\kd(x,z)$ for
the minimal value of $k$ above, and we set
$\kd(x,x)=0$.
Let the \df{cluster} at $x\in\H$ be the set
$$
C_x:= \{y: x\klra y\},
$$
and let its `radius' be
$$
\rho_K(x) = \sup\{\kd(x,z): z \in C_x\}.
$$

We seek conditions under which $\P(C_0<\oo)=1$ or, stronger, $\rho_K(0)$ and
$|C_0|$ have stretched-exponential tails.

\begin{thm}\label{thm:hall2}
Let $d \ge 2$.
\begin{letlist}
\item
If $J$ is $0$-monotone and $\mu:=\E|K_0| -1$
satisfies $\mu<1$, then
$\E|C_0| \le (1-\mu)^{-1}$ and
$$
\P(\rho_K(0) \ge n) \le \mu^n, \qq n \ge 0.
$$
\item
If, in addition,
there exist $\zeta>0$ and $a\in(0,\frac12)\cup\{1\}$ such that
$$
\P(|K_0|\ge n) \le \exp(-\zeta n^a), \qq n \ge 2,
$$
then there exists $\zeta' \in(0,\oo)$ such that
$$
\P(|C_0| \ge n) \le \exp(-\zeta' n^a), \qq n \ge 2.
$$
\end{letlist}
\end{thm}

This theorem is related to certain results for continuum percolation to be
found in \cite{Gou,MR}. The proof of part (b) will make use of the following
theorem for the tail of the total progeny in a branching process with
stretched-exponential family-size distribution.

\begin{thm}\label{thm:bptail}
Let $T$ be the total progeny in a branching process
with typical family-size $F$ satisfying $\E F < 1$ and
\begin{equation}\label{assum1}
\P(F>n) \le \exp(-\gamma n^a),\qq n \ge 2,
\end{equation}
for constants $\gamma\in(0,\oo)$ and $a\in(0,\frac12)\cup\{1\}$.
There exists $\gamma'\in(0,\oo)$ such that
$$
\P(T>n) \le \exp(-\gamma' n^a), \qq n \ge 2.
$$
\end{thm}

\begin{proof}[Proof of Theorem \ref{thm:hall2}]
Suppose that $J$ is $0$-monotone, and let $\H$ be
ordered in some arbitrary but fixed manner starting with the
origin. We shall construct the cluster $C_0$ in an iterative
manner, and shall compare certain features of $C_0$ with those
of a branching process. This branching process will be
subcritical if $\mu := \E|K_0\sm\{0\}| < 1$, and the claims will
follow by Theorem \ref{thm:bptail}.

At stage $0$, we write $l_0:=0$; the \df{family} of $l_0$ is
defined to be the set $F_0:=K_{l_0}\sm\{l_0\}$, and we declare
$l_0$ to be \df{dead} and sites in $F_0$ to be \df{live}. At
stage $1$, we let $l_1$ be the earliest live site, and define
its family as the set $F_{1}:=K_{l_1}\sm F_0$; we declare sites
in $F_1$ to be live, and $l_1$ to be \df{dead}. Suppose that,
after stage $n-1$, we have defined the families
$F_0,F_1,\dots,F_{n-1}$, and have a current live set $G_{n-1}$
and dead set $D_{n-1}$. At stage $n$, we let $l_{n}$ be a live
site (chosen in a way that we describe next), and declare
\begin{align*}
F_{n} &:= K_{l_n}\sm (\{l_0\}\cup F_0\cup\dots\cup F_{n-1}),\\
G_n &:=(G_{n-1}\cup F_n)\sm \{l_n\},\\
D_n &:=D_{n-1}\cup\{l_n\}.
\end{align*}
The site $l_n$ is chosen as follows. Let $m:=\min\{k: F_k \cap
G_{n-1} \ne \es\}$, the index of the earliest family containing
a live site, and let $l_n$ be the earliest live site in
$F_m$.

This process either terminates or continues forever. In the
former case, let $N$ be the greatest value of $n$ for which $F_n$
is defined, and set $N=\oo$ in the latter case. It is easily
seen that $C_0=\bigcup_{n=0}^N F_n$.

We claim that the above process is dominated stochastically by
a branching process with family-sizes distributed as
$X:=|K_0\sm\{0\}|$, and we explain this by a recursive
argument. The size of the family $F_0$ of $l_0$ is evidently
distributed as $X$, and is given in terms of the sequence $J$
as follows. Let
$$
Y_0:= \{y\in \H: l_0 \in J_y\}.
$$
Then $F_0=\left(\bigcup_{y\in Y_0} J_y\right)\sm\{l_0\}$.  In
determining the family $F_1$ of $l_1$, we set
$$
Y_1:= \{y \in \H \sm Y_0: l_1 \in J_y\},
$$
and we have that
$$
F_1= \left(\bigcup_{y\in Y_1} J_y\right)\sm (\{l_0\}\cup F_0).
$$
Since $J$ is $0$-monotone, given $Y_0$ the family
$(J_y: y\notin Y_0)$ is stochastically dominated by a family
$(J_y': y \notin Y_0)$ of independent random sets such that
$J_y'$ has the distribution of $J_y$. It follows that, given
$F_0$, the conditional distribution of $|F_1|$ is no greater than
that of $X$.

Let $n \ge 1$. Suppose stage $n-1$ is complete, and write
\begin{align*}
Y_{i} &:= \bigl\{y\in\H\sm(Y_0\cup\dots\cup Y_{i-1}): l_i \in J_y\bigr\},\\
F_i &:=
\left(\bigcup_{y\in Y_i} J_y\right)\sm (\{l_0\}\cup F_0\cup\dots\cup F_{i-1}),
\end{align*}
for $0\le i \le n$. Given $(l_i,Y_i,F_i)$ for $0\le i < n$, the
set $(J_y: y \notin Y_0\cup\dots\cup Y_{n-1})$ is
stochastically dominated by an independent family $(J_y':
y\notin Y_0\cup\dots\cup Y_{n-1})$ where each $J'_y$ has the
unconditional law of $J_y$. Therefore, the conditional law of
$|F_n|$ is stochastically smaller than that of $X$.  This
completes the proof of domination by a branching process.

The dominating branching process has mean family-size $\mu:=
\E|K_0\sm\{0\}|$ which, by assumption, satisfies $\mu<1$. The
process is therefore subcritical. It follows in particular that
$\E|C_0| \le (\E|K_0|-1)^{-1}$, and $\P(|C_0|<\oo) = 1$. The
radius $\rho_K(0)$ of $C_0$ is stochastically smaller than the
number of generations of the branching process, so that
$$
\P(\rho_K(0) \ge n) \le  \mu^n,
$$
and part (a) of the theorem is proved.

Part (b) is a consequence of Theorem \ref{thm:bptail}.
\end{proof}

\begin{proof}[Proof of Theorem \ref{thm:bptail}]
The total progeny of a branching process is connected to
the length of a busy period of a certain queue, and to a first-passage
time of a certain random walk. One may construct a proof
that exploits the connection to queueing theory and makes use of
\cite[Thm 1.2]{BDK}, but instead we shall use the representation
in terms of random walks.

Let $X$ be a random variable taking values in the non-negative integers,
such that $\E X<1$.
Consider a random walk $(S_n: n \ge 0)$ on the integers with
$S_0=1$ and steps
distributed as $X-1$.
Let
$$
T:= \inf\{n: S_n = 0\}.
$$
It is
standard that $T$ has the same distribution as the size
of the total progeny of a branching process with family-sizes
distributed as $X$. The relationship between the branching process and
the random walk is as follows. Suppose the elements of
the branching process are ordered within families in some arbitrary way.
When an element having a family of size $F$ arrives, the walker
is displaced by a step $F-1$.

Let $F$ satisfy $\E F < 1$ and \eqref{assum1} with
some $\gamma >0$ and $a\in(0,\frac12)$.
Let $F'$ be a random variable taking values in the non-negative integers
such that $\E F' < 1$ and
\begin{equation}\label{now42}
\P(F' \ge n) = \exp(-\gamma n^a), \qq n \ge N,
\end{equation}
for some $N \ge 1$. Let $T'$ be the first-passage time to $0$ of
a random walk with steps distributed as $F'-1$, starting
at $1$.

If $0<a < \frac12$, by \cite[Thms 8.2.3--8.2.4]{BB},
 there exists $c \in (0,\oo)$ such that
\begin{equation}\label{now41}
\P(T' \ge n) \sim c \exp\bigl(-[1- \E F']n\bigr)\q \text{as} \q n\to\oo.
\end{equation}

Let $0<\eps < 1-\E F$.
It is an elementary exercise to find a positive integer $N$ and
a random variable $F'$, taking values in the non-negative integers, such that:
\begin{letlist}
\item $F$ is stochastically smaller than $F'$,
\item $\E F' \le \E F + \eps$,
\item \eqref{now42} holds.
\end{letlist}
With $F'$ chosen thus, we have $\E F' < 1$ by (b).
Therefore, the corresponding first-passage time $T'$ satisfies \eqref{now41}.
The claim of the theorem follows by the fact that
$T$ is stochastically smaller than $T'$.

When $a=1$, basically the same argument is valid. When
the moment generating function $M(\theta) = \E(e^{\theta F'})$
satisfies
$$
\theta_0 := \sup\{\theta:M(\theta)<\oo\}>0,
$$
quite precise estimates are known for the tail of $T'$ in terms
of the minimum of $M$. Complications arise when the minimum is
achieved at $\theta=\theta_0$, as discussed in \cite[Sect.\
8.2.3]{BB}. We avoid these details here by citing \cite[Thm
1]{Hey64} in place of \cite[Thms 8.2.3--8.2.4]{BB} above.
\end{proof}

\begin{proof}[Proof of Theorem \ref{thm:noperc}]
Let $\pM<p<1$, implying in particular that $\rho<1$.
Let $J=(J_y: y \in \H)$ be independent subsets
of $\H$ such that, for each $y$, $J_y$ has the same law as
$L(H_y)$. The corresponding sets $K_x$ are given as in \eqref{defK}.
We shall apply Theorem \ref{thm:hall2} to
the sequence $J$.

By Lemma \ref{lem1},
$\E|L(M_0)| \le 1+\s(d,p)$. Since $J_0$ has the same law as
$L(H_0)$, and $\rad(H_0) = \rad(L(H_0))$, the proof of
\eqref{now26} may be followed with $M_0$ and $H_y$ replaced by
$K_0$ and $J_y$ respectively. This yields
\begin{equation}\label{now26b}
\P_p(\rad(K_0)\ge r) \le cr^{d-1} \frac {\rho^{r/2}}{1-\rho},\qq r\ge 1,
\end{equation}
and hence $\mu := \E|K_0|-1$ satisfies
$$
\mu\le \s(d,p) < 1.
$$

We claim that $J$
is $0$-monotone. Let $S \subseteq \H$ with $0\notin
S$, and let $A \subseteq \H$. The events $\{L(H_0) \subseteq
A\}$ and $\{L(H_0) \cap S = \es\}$ are increasing, and
therefore
\begin{equation}\label{eqn:dec}
\P_p(L(H_0) \subseteq A \mid L(H_0) \cap S = \es)
 \ge \P_p(L(H_0)\subseteq A),
\end{equation}
by the FKG inequality. The claim follows.

By either Theorem \ref{thm:indbnd0} or Theorem
\ref{thm:indbnd}, $|R_0|$ and $\rho_M(0)$ are bounded above
(stochastically) by $|C_0|$ and $\rho_K(0)$. The claim of part (a) follows
by Theorem \ref{thm:hall2}(a).

Let $d\ge 2$.  The proof of Corollary
\ref{cor:subexp}(b) holds with $M_0$ and $L(M_0)$ replaced by
$K_0$, and with \eqref{now26b} in place of \eqref{now26}, yielding that
\begin{equation}\label{now34}
\P(|K_0| \ge n) \le \exp(-\beta n^{1/(d-1)}), \qq n \ge 2,
\end{equation}
where $\beta>0$.

When $d \ne 3$,
claim (b) holds by \eqref{now34} and Theorem \ref{thm:hall2}(b).
Let $d=3$.
By \eqref{now34}, for all $a\in(0,\frac12)$
there exists $\zeta>0$ such that
$$
\P_p(|K_0| \ge n) \le \exp(-\zeta n^a), \qq n \ge 2,
$$
and claim (c) follows by Theorem \ref{thm:hall2}(b).
\end{proof}

\section*{Acknowledgements}
GRG acknowledges support from Microsoft Research during his
stay as a Visiting Researcher in the Theory Group in Redmond.
This work was completed during his attendance at a programme at
the Isaac Newton Institute, Cambridge. Daryl Daley and Sergey
Foss kindly advised the authors on the literature concerning
random walks with heavy-tailed steps.

\bibliographystyle{amsplain}
\bibliography{hill}

\end{document}